\documentclass[12pt]{amsart}
\allowdisplaybreaks[2]
\usepackage{amssymb}
\usepackage{graphicx}
\usepackage{hyperref}
\usepackage{color}
\numberwithin{equation}{section}
\theoremstyle{plain}
\newtheorem{thm}{Theorem}[section]
\newtheorem{theorem}[thm]{Theorem}

\theoremstyle{definition}

\newtheorem{definition}[thm]{Definition}

\makeatletter

\newcommand{\Rmnum}[1]{\extendafter\@slowromancap\romannumeral#1@}
\makeatother

\newcommand{\tabincell}[2]{\begin{tabular}{@{}#1@{}}#2\end{tabular}}

\begin{document}
\title[Numerical Computations For Operator Axioms]
{Numerical Computations For Operator Axioms} %
\author[Pith Peishu Xie]{Pith Peishu Xie}
\address{Axiom Studio \\ PO Box \#3620 \\ Jiangdongmen Postoffice \\ Gulou District %
\\ Nanjing, 210036, P.R. China} %
\email{pith.xie@gmail.com} %
\begin{abstract}
The Operator axioms have produced new real numbers with new operators. New operators naturally produce new equations and thus extend the traditional mathematical models which are selected to describe various scientific rules. So new operators help to describe complex scientific rules which are difficult described by traditional equations and have an enormous application potential. As to the equations including new operators, engineering computation often need the approximate solutions reflecting an intuitive order relation and equivalence relation. However, the order relation and equivalence relation of real numbers are not as intuitive as those of base-b expansions. Thus, this paper introduces numerical computations to approximate all real numbers with base-b expansions.
\end{abstract}
\maketitle
\setcounter{tocdepth}{5} \setcounter{page}{1}
%

\section{Introduction}
In \cite{Ref1}, we distinguished the limit from the infinite sequence. In \cite{Ref2},  we defined the Operator axioms to extend the traditional real number system. In \cite{Ref3},  we promoted the research in the following areas:

1. We improved on the Operator axioms. 

2. We defined the VE function and EV function. For clarity, we rename VE Function\cite{Ref3} to Prefix Function. For clarity, we rename EV Function\cite{Ref3} to Suffix Function. 

3. We proved two theorems about the Prefix function and Suffix function.

The Operator axioms forms a new arithmetic axiom. The \cite[Definition 2.2]{Ref3} defines real number system on the basis of the logical calculus $ \{ \Phi, \Psi \} $. The \cite[TABLE 2]{Ref2} defines new operators according to the definition of number systems. Real operators naturally produce new equations such as $ y = [x++++[1+1]] $, $ y = [[1+1]----x] $, $ y = [x////[1+1]] $ and so on. In other words, real operators extend the traditional mathematical models which are selected to describe various scientific rules. Operator axioms have included infinite operators, so no other operator can be added to them. This means the Operator axioms is a complete real number system. In fact, infinite operators imply the completeness.

Thus, real operators exhibit potential value as follows:

1. Real operators can give new equations and inequalities so as to precisely describe the relation of mathematical objects.

2. Real operators can give new equations and inequalities so as to precisely describe the relation of scientific objects.

So real operators help to describe complex scientific rules which are difficult described by traditional equations and have an enormous application potential.

As to the equations including real operators, engineering computation often need the approximate solutions reflecting an intuitive order relation and equivalence relation. Although the order relation and equivalence relation of real numbers are consistent, they are not as intuitive as those of base-b expansions. In practice, it is quicker to determine the order relation and equivalence relation of base-b expansions. So we introduce numerical computations to approximate real numbers with base-b expansions.

The numerical computations we proposed are not the best methods to approximate real numbers with base-b expansions, but the simple methods to approximate real numbers with base-b expansions. The compution complexity of the numerical computations we proposed could be promoted furtherly. However, we first prove that the Operator axioms can run on any modern computer. The numerical computation we proposed blends mathematics and computer science. Modern science depends on both the mathematics and the computer science. Arithmetic is the core of both the mathematics and the science. As a senior arithmetic, the Operator axioms will promote both the mathematics and the science in the future.

\begin{theorem}\label{Them_SUM}
Any positive number $ \xi $ may be expressed as a limit of an infinite base-b expansion sequence %
\begin{eqnarray}
\lim\limits_{n \to \infty} A_{1} A_{2} \cdots A_{s+1}.\ a_1 a_2 a_3 \cdots a_n, %
\end{eqnarray}

where $ 0 \leq A_{1} < b, 0 \leq A_{2} < b, \cdots , 0 \leq a_{n} < b $, %
not all A and a are 0, and an infinity of the $ a_n $ are less than (b-1). If $ \xi \geq 1 $, %
then $ A_{1} \geq 0 $.
\end{theorem}

\begin{proof}
Let $ [\xi] $ be the integral part of $ \xi $. Then we write %
\begin{eqnarray}\label{1.1}
\xi = [\xi] + x = X + x,
\end{eqnarray}
where $ X $ is an integer and $ 0 \leq x < 1 $, and consider $ X $ and $ x $ separately. %

If $ X > 0 $ and $ b^{s} \leq x < b^{s+1} $, %
and $ A_1 $ and $ X_1 $ are the quotient and remainder when $ X $ is divided by $ b^{s} $, %
then $ X = A_1 \cdot b^{s} + X_1 $, where $ 0 < A_1 = [b^{-s}X] < b $, %
$ 0 \leq X_1 < b^s $. %

Similarly %
\begin{eqnarray*}
X_1 = & A_2 \cdot b^{s-1} + X_2 & (0 \leq A_2 < b, 0 \leq X_2 < b^{s-1}), \\
X_2 = & A_3 \cdot b^{s-2} + X_3 & (0 \leq A_3 < b, 0 \leq X_3 < b^{s-2}), \\
\cdots & \cdots & \cdots \\
X_{s-1} = & A_s \cdot b + X_s & (0 \leq A_s < b, 0 \leq X_s < b), \\
X_s = & A_{s+1} & (0 \leq A_{s+1} < b).
\end{eqnarray*}
Thus $ X $ may be expressed uniquely in the form %
\begin{eqnarray}
X = A_1 \cdot b^s + A_2 \cdot b^{s-1} + \cdots + A_s \cdot b +
A_{s+1},
\end{eqnarray}
where every $ A $ is one of 0, 1, $ \cdots $, (b-1), and $ A_1 $ is not 0. We abbreviate %
this expression to %
\begin{eqnarray}\label{1.3}
X = A_1 A_2 \cdots A_s A_{s+1},
\end{eqnarray}
the ordinary representation of $ X $ in base-b expansion notation. %

Passing to $ x $, we write %
\begin{eqnarray*}
& X = f_1 & (0 \leq f_1 < 1).
\end{eqnarray*}
We suppose that $ a_1 = [b f_1] $, so that %
\begin{eqnarray*}
\frac{a_1}{b} \leq f_1 < \frac{a_1 + 1}{b};
\end{eqnarray*}
$ a_1 $ is one of 0, 1, $ \cdots $, (b-1), and %
\begin{eqnarray*}
a_1 = [b f_1], & b f_1 = a_1 + f_2 & (0 \leq f_2 < 1).
\end{eqnarray*}

Similarly, we define $ a_2, a_3, \cdots $ by %
\begin{eqnarray*}
a_2 = [b f_2], & b f_2 = a_2 + f_3 & (0 \leq f_3 < 1), \\
a_3 = [b f_3], & b f_3 = a_3 + f_4 & (0 \leq f_4 < 1), \\
\cdots & \cdots & \cdots
\end{eqnarray*}
Every $ a_n $ is one of 0, 1, $ \cdots $, (b-1). Thus %
\begin{eqnarray}\label{1.4}
x = x_n + g_{n+1},
\end{eqnarray}
where
\begin{eqnarray}
x_n = \frac{a_1}{b} + \frac{a_2}{b^2} + \cdots + \frac{a_n}{b^n}, \label{1.5} \\
0 \leq g_{n+1} = \frac{f_{n+1}}{b^n} < \frac{1}{b^n}.
\end{eqnarray}

We thus define a base-b expansion $ .a_1 a_2 a_3 \cdots a_n \cdots $ associated with $ x $. We call %
$ a_1, a_2, \cdots $ the first, second, $ \cdots $ \emph{digits} of the base-b expansion. %

Since $ a_n < b $, the series %
\begin{eqnarray}\label{1.7}
\sum\limits_{1}^{\infty} \frac{a_n}{b^n}
\end{eqnarray}
is convergent; and since $ g_{n+1} \rightarrow 0 $, its sum is $ x $. We may therefore write %
\begin{eqnarray}\label{1.8}
x = .\ a_1 a_2 a_3 \cdots,
\end{eqnarray}
the right-hand side being an abbreviation for the series (\ref{1.7}). %

We now combine (\ref{1.1}), (\ref{1.3}), and (\ref{1.8}) in the form %
\begin{eqnarray}
\xi = X + x = A_1 A_2 \cdots A_s A_{s+1}.\ a_1 a_2 a_3 \cdots; %
\end{eqnarray}
and the claim follows. %
\end{proof}

\textbf{Theorem \ref{Them_SUM}} implies that every real number has base-b expansions arbitrary close to it. So in numerical computations for the Operator axioms, all operands and outputs are denoted by base-b expansions to intuitively show the order relation and equivalence relation.

The paper is organized as follows. In Section 2, we define the operation order for all operations in the Operator axioms. In Section 3, we construct the numerical computations for binary operations. In Section 4, we define some concepts in the Operator axioms.

\section{Operation Order}
In the Operator axioms, the number `1' is the only base-b expansion while the others derive from the operation of two numbers and one operator. For example, the number ``$[[1+[1+1]]----[1+1]]$" derives from the operation of the number ``$[1+[1+1]]$", the number ``$[1+1]$" and the real operator ``$----$".

\begin{definition}
\emph{Numerical computation} is a conversion from an operation to an approximate base-b expansion.
\end{definition}

In general, an operation includes many binary operators. For example, the operation ``$[[1+[1+1]]----[1+1]]$" includes three ``$+$" and one ``$----$". Since each operator produces a binary operation, $n$ operator in an operation will produce $n$ binary operations. It is better to compute all binary operations in an operation in order. The order is denoted as \emph{Operation Order}.

\cite[\S5.3.1]{Ref4} stores tradition operations as an expression tree and then applies traversal algorithm to evaluate the expression tree. Likewise, each operation of the Operator axioms can be stored as an expression tree in which each number `1'  become a leaf node and each operator become an internal node. Then Operation Order is just the traversal order of the expression tree. In this paper, we choose inorder traversal as Operation Order. Figure \ref{tree} illustrates an expression tree for the number ``$[[1+[1+1]]----[1+1]]$".

\begin{figure}
\centering
\includegraphics[width=\textwidth]{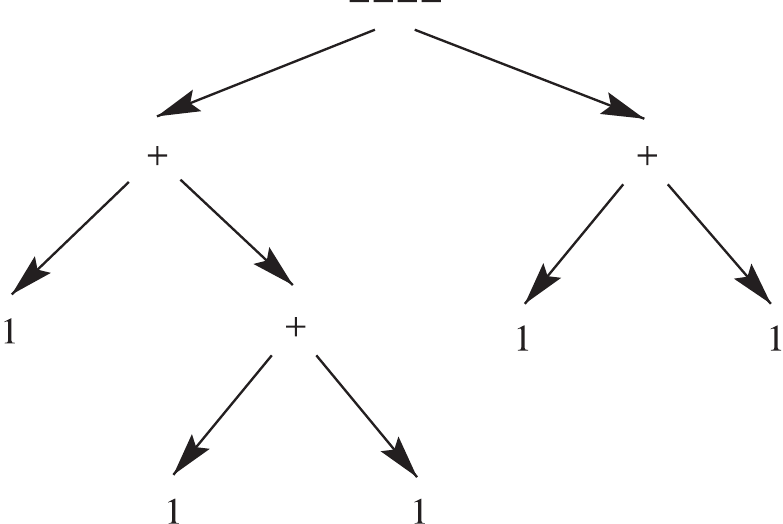}
\caption{An Expression Tree For The Number ``$[[1+[1+1]]----[1+1]]$"}
\label{tree}
\end{figure}

It is supposed that the numerical computation applies base-10 expansions. Then the numerical computation for ``$[[1+[1+1]]----[1+1]]$" will proceed with the following Operation Order:
\begin{eqnarray*}
[[1+[1+1]]--[1+1]] & = & [[1+2]----[1+1]] \\
& = & [3----[1+1]] \\
& = & [3----2]
\end{eqnarray*}

In summary, every operation in the Operator axioms can divide into many binary operations by Operation Order. So numerical computations focus on the binary operations.

\section{Numerical Computations For Binary Operations}
In \cite{Ref3}, Operator axioms have expressed the real number system. In this section, ``real number" refer to the real number deduced from Operator axioms\cite{Ref3}.

\subsection{Division Of Binary Operations}

According to the complexity of numerical computations, we divide binary operations into low operations, middle operations and high operations. \autoref{table:division} lists their elements in detail.

\begin{table}[h]
\begin{center}
\centering \caption{Division Of Binary Operations}
\label{table:division}
\begin{tabular}{|c|c|c|c|}
\hline
& \textbf{Low Operations} & \textbf{Middle Operations} & \textbf{High Operations} \\
\hline
\textbf{Operators} & $ +, ++, -, --, /, // $ & $ +++, ---, /// $ & \tabincell{c}{$ ++++, +++++, \cdots, $ \\
$ ----, -----, \cdots, $ \\
$ ////, /////, \cdots $} \\
\hline
\end{tabular}
\end{center}
\end{table}

\subsection{Numerical Computations For Low Operations}
In the Operator axioms, $ / $ is equal to $ - $ while $ // $ is equal to $ -- $. From a traditional viewpoint, the low operations ``$ +, ++, -, -- $" are equal to basic arithmetic operations ``$ +, \times, -, \div $". So numerical computations for low operations have been constructed in elementary arithmetic.

\subsection{Numerical Computations For Middle Operations}\label{mid}
From a traditional viewpoint, $ +++ $ is an exponentiation operation, $ --- $ is a root-extraction operation and $ /// $ is a logarithm operation. In this subsection, we import the numerical computations for the middle operations ``$ +++, ---, /// $" in \cite[\S 23]{Ref5}.

Let $e$ be Euler's number. It is supposed that $ a \in \bigl(-\infty, +\infty\bigl) $ is a base-b expansion, $ n  \in Z $ and $ k \in N $. The numerical computation for $ [e+++a] $ can be constructed with the Taylor-series expansion as follows.
\begin{eqnarray*}
[e+++a] & = & \left[ \lim\limits_{k \to +\infty} \left( \sum\limits_{n=0}^{k} [[a+++n]--[n!]] \right) \right] \\
& \approx & [1 + [a--[1!]] + [[a+++2]--[2!]] + [[a+++3]--[3!]] + \cdots + \\
&& [[a+++n]--[n!]] + \cdots]
\end{eqnarray*}

It is supposed that $ a \in \bigl(0, +\infty\bigl) $ is a base-b expansion and $ n  \in Z $. Let $ b=[[a-1]--[a+1]] $, then the numerical computation for $ [a///e] $ can be constructed with the Taylor-series expansion as follows.
\begin{eqnarray*}
[a///e] & = &[[[1+b]--[1-b]]///e] \\
& = & \left[ 2++\left[ \lim\limits_{k \to +\infty} \left( \sum\limits_{n=0}^{k} [[b+++[2n+1]]--[2n+1]] \right) \right] \right] \\
& \approx & [2++[b + [[b+++3]--3] + [[b+++5]--5] + \cdots + \\
&& [[b+++[2n+1]]--[2n+1]] + \cdots]]
\end{eqnarray*}

It is supposed that $a$ and $b$ are two base-b expansions, where $ a \in \bigl(0, +\infty\bigl) $ and $ b \in \bigl(-\infty, +\infty\bigl) $. Then the numerical computation for $ [a+++b] $ can be divided and conquered with the identity $ [a+++b] = [e+++[b++[a///e]]] $.

It is supposed that $ [a+++b] $ is a real number, where $ a \in \bigl(-\infty, 0\bigl] $ and $ b \in \bigl(-\infty, +\infty\bigl) $ are two base-b expansions. Then the numerical computation for $ [a+++b] $ can always be equated with the basic numerical computations as above and the basic arithmetic operations in the axioms (OA.1)$\sim$(OA.75).

\subsection{Numerical Computations For High Operations}
It is supposed that the constants $ n, a_1, a_2, b_1, b_2, c_1, c_2, m \in N $.

Zero can be equated with a fraction $ [[1-1]--1] $. Any non-zero base-b expansion can be equated with a fraction $ [[1-1]\pm[a_1--a_2]] $ for $ a_1, a_2 \in N $.

It is supposed that the constant $ d \in R $ with $ [1-1] < d $. It is supposed that the constant $ e \in R $ with $ 1 < e $. For clarity, we rename VE Function\cite{Ref3} to Prefix Function. For clarity, we rename EV Function\cite{Ref3} to Suffix Function. 

\begin{definition}\cite[Definition 3.1]{Ref3}
\emph{Prefix Function} is the function $ f : \bigl[1, +\infty\bigl) \rightarrow R $ defined by $ f(x)=[x+++\dot{e}d] $.
\end{definition}

\begin{definition}\cite[Definition 3.2]{Ref3}
\emph{Suffix Function} is the function $ f : R \rightarrow R $ defined by $ f(x)=[e+++\dot{e}x] $.
\end{definition}

\begin{definition}\cite[Definition 3.3]{Ref3}
\emph{Fundamental operator functions} are \emph{Prefix Function} and \emph{Suffix Function}.
\end{definition}

\begin{theorem}\cite[Theorem 3.4]{Ref3}\label{Prefix}
The Prefix Function $ f(x)=[x+++\dot{e}d] $ is continuous, unbounded and strictly increasing.
\end{theorem}

\begin{theorem}\cite[Theorem 3.5]{Ref3}\label{Suffix}
The Suffix Function $ f(x)=[e+++\dot{e}x] $ is continuous, unbounded and strictly increasing.
\end{theorem}

It is supposed that the constant $ t, u, v \in R $.

\begin{definition}
\emph{Root Equations} are the equations such as $f(x)=t$ such that:
\begin{enumerate}
\item The function $f(x)$ is real and continuous on any closed interval $ \bigl[u, v\bigl] $ in the domain;
\item The equation $f(x)=t$ has only one root on the above interval $ \bigl[u, v\bigl] $;
\end{enumerate}
\end{definition}

\cite[TABLE PT2.3]{Ref6} lists common root-finding methods and their convergence conditions. When $a$ acts as the lower guess and $b$ acts as the upper guess, both the bisection method\cite[\S 5.2]{Ref6} and Brent's method\cite[\S 6.4]{Ref6} always converge and find the only root on $ \bigl[u, v\bigl] $. But Brent's method converges faster than the bisection method and thus acts as the main root-finding method for Root Equations.

In the following, we construct all numerical computations for high operations by induction. It is supposed that the constant $ p_1 \in R $ with $ 1 \leq p_1 $. It is supposed that the constant $ p_2 \in R $. It is supposed that the constant $ q_1 \in R $ with $ 1 \leq q_1 $. It is supposed that the constant $ q_2 \in R $ with $ [1-1] \leq q_2 $. It is supposed that the constant $ r_1 \in R $ with $ [1-1] < r_1 $. It is supposed that the constant $ r_2 \in R $ with $ 1 < r_2 $. In the following, we approximate these constants with those fractions such as $[a_1--a_2]$.

\begin{enumerate}\setlength{\parindent}{2em}
\item The numerical computations for the operations $ [p_1++++p_2] $, $ [q_1----q_2] $, $ [r_1////r_2] $ are constructed.

\begin{enumerate}\setlength{\parindent}{2em}
\item A numerical computation for $ [[a_1--a_2]++++[b_1--b_2]] $ with $ 1 \leq [a_1--a_2] $ and $ [1-1] \leq [b_1--b_2] $ can be constructed.

\begin{enumerate}\setlength{\parindent}{2em}
\item $ [b_1--b_2] \leq 1 $.
\begin{eqnarray*}
(A1) & (\bar{c} = [[[a_1--a_2]-1]++[1+[1+1]]]) \wedge  & \\
& (\bar{d} = [[[a_1--a_2]-1]++[1+1]]) \wedge \\
& (\bar{e} = [[b_1--b_2]+++[1+1]]) \wedge  & \\
& (\bar{f} = [[b_1--b_2]+++[1+[1+1]]]) \wedge & \\
& ([[a_1--a_2]++++[b_1--b_2]] = & \\
& [[1+[\bar{c}++[\bar{e}+++1]]]-[\bar{d}++[\bar{f}+++1]]]) & \text{by (OA.83),(OA.28)} \\
(A2) & (\bar{c} = [[[a_1--a_2]-1]++[1+[1+1]]]) \wedge  & \\
& (\bar{d} = [[[a_1--a_2]-1]++[1+1]]) \wedge \\
& (\bar{e} = [[b_1--b_2]+++[1+1]]) \wedge  & \\
& (\bar{f} = [[b_1--b_2]+++[1+[1+1]]]) \wedge & \\
& ([[a_1--a_2]++++[b_1--b_2]] = & \\
& [[1+[\bar{c}++\bar{e}]]-[\bar{d}++\bar{f}]]) & \text{by (OA.83),(OA.62)} \\
\end{eqnarray*}

\item $ 1 < [b_1--b_2] $.
\begin{eqnarray*}
(B1) & [[a_1--a_2]++++[b_1--b_2]] = & \\
& [[a_1--a_2]+++[[a_1--a_2]++++[[b_1--b_2]-1]]] & \text{by (OA.82)} \\
(B2) & \text{It is supposed that $ [1-1] \leq [[b_1--b_2]-n] $ and } & \\
& \text{$ [[b_1--b_2]-n] \leq 1 $. Let us distinguish these } & \\
& \text{$[a_1--a_2]$ with the subscripts } \{(1), (2), (3), \cdots\}. & \\
(B3) & \Rightarrow [[a_1--a_2]++++[b_1--b_2]] = & \\
& [[a_1--a_2]_{(1)}+++[[a_1--a_2]_{(2)}+++ \cdots \\
& [[a_1--a_2]_{n}+++
[[a_1--a_2]++++ & \\
& [[b_1--b_2]-n]] \cdots ]] & \text{by (OA.82),(B2)} \\
(B4) & (\bar{c} = [[[a_1--a_2]-1]++[1+[1+1]]]) \wedge  & \\
& (\bar{d} = [[[a_1--a_2]-1]++[1+1]]) \wedge \\
& (\bar{e} = [[[b_1--b_2]-n]+++[1+1]]) \wedge  & \\
& (\bar{f} = [[[b_1--b_2]-n]+++[1+[1+1]]]) \wedge & \\
& ([[a_1--a_2]++++[[b_1--b_2]-n]] = & \\
& [[1+[\bar{c}++\bar{e}]]-[\bar{d}++\bar{f}]]) & \text{by (B2),(A2)} \\
(B5) & \Rightarrow (\bar{c} = [[[a_1--a_2]-1]++[1+[1+1]]]) \wedge  & \\
& (\bar{d} = [[[a_1--a_2]-1]++[1+1]]) \wedge \\
& (\bar{e} = [[[b_1--b_2]-n]+++[1+1]]) \wedge  & \\
& (\bar{f} = [[[b_1--b_2]-n]+++[1+[1+1]]]) \wedge & \\
& ([[a_1--a_2]++++[b_1--b_2]] = & \\
& [[a_1--a_2]_{(1)}+++[[a_1--a_2]_{(2)}+++ \cdots \\
& [[a_1--a_2]_{n}+++
[[1+[\bar{c}++\bar{e}]]-[\bar{d}++\bar{f}]] \cdots ]]) & \text{by (B3),(B4)} \\
\end{eqnarray*}

Then (A1)$\sim$(A2) and (B1)$\sim$(B4) have reduced one $ ++++ $ operation to many $ +++ $ operations. Since \S \ref{mid} has constructed the numerical computation for $ [[a_1--a_2]++++[b_1--b_2]] $, (A1)$\sim$(A2) and (B1)$\sim$(B4) can achieve a numerical computation for $ [[a_1--a_2]++++[b_1--b_2]] $.
\end{enumerate}

\item A numerical computation for $ [[a_1--a_2]----[b_1--b_2]] $ with $ 1 \leq [a_1--a_2] $ and $ [1-1] < [b_1--b_2] $ can be constructed.

The numerical computation for $ [[a_1--a_2]----[b_1--b_2]] $ is equated with the numerical root finding of the equation $ x=[[a_1--a_2]----[b_1--b_2]] $.
\begin{eqnarray*}
(A1) & x=[[a_1--a_2]----[b_1--b_2]] & \\
(A2) & \Rightarrow [x++++[b_1--b_2]]= & \\
& [[[a_1--a_2]----[b_1--b_2]] & \\
& ++++[b_1--b_2]] & \text{by (OA.108)} \\
(A3) & \Rightarrow [x++++[b_1--b_2]]=[a_1--a_2] & \text{by (OA.79),(OA.24),} \\ 
&& \text{(OA.25)}
\end{eqnarray*}

According to (OA.37)$\sim$(OA.40), the function $ f(x)=[x++++[b_1--b_2]] $ is defined on the domain $ \bigl[1, +\infty\bigl) $. Theorem \ref{Prefix} implies that we can iteratively increase $v$ by step 1 from $v=1$ until $ [a_1--a_2] < [v++++[b_1--b_2]] $ holds. 

Theorem \ref{Prefix} implies that $ f(x)=[x++++[b_1--b_2]] $ is continuous on the domain $ \bigl[1, v\bigl] $. (OA.75) derives that $ [1++++[b_1--b_2]] = 1 $. So (OA.95) derives that $ [1++++[b_1--b_2]] \leq [a_1--a_2] $. In summary, both $ [1++++[b_1--b_2]] \leq [a_1--a_2] $ and $ [a_1--a_2] < [v++++[b_1--b_2]] $ hold.

Then Intermediate Value Theorem derives that the equation $ [x++++[b_1--b_2]]=[a_1--a_2] $ has only one root on the domain $ \bigl[1, v\bigl] $. Since Theorem \ref{Prefix} implies that the equation $ [x++++[b_1--b_2]]=[a_1--a_2] $ has no root on the domain $ \bigl(v, +\infty\bigl) $, the equation $ [x++++[b_1--b_2]]=[a_1--a_2] $ has only one root on the domain $ \bigl[[1-1], +\infty\bigl) $. Since the equation $ [x++++[b_1--b_2]]=[a_1--a_2] $ belongs to Root Equations, Brent's method can find the only root of the equation and constructs the numerical computation for $ [[a_1--a_2]----[b_1--b_2]] $.

\item A numerical computation for $ [[a_1--a_2]////[b_1--b_2]] $ with $ 1 \leq [a_1--a_2] $ and with $ 1 < [b_1--b_2] $ can be constructed.

The numerical computation for $ [[a_1--a_2]////[b_1--b_2]] $ is equated with the numerical root finding of the equation $ x=[[a_1--a_2]////[b_1--b_2]] $.
\begin{eqnarray*}
(A1) & x=[[a_1--a_2]////[b_1--b_2]] & \\
(A2) & [1-1] < [[a_1--a_2]////[b_1--b_2]] & \text{by (OA.33)} \\
(A3) & \Rightarrow 1 < [[b_1--b_2]++++ & \\
& [[a_1--a_2]////[b_1--b_2]]] & \text{by (OA.37),(OA.72)} \\
(A4) & \Rightarrow ([[b_1--b_2]++++ & \\
& [[a_1--a_2]////[b_1--b_2]]]) & \text{by (OA.90)} \\
(A5) & \Rightarrow [[b_1--b_2]++++x]= & \\
& [[b_1--b_2]++++ & \\
& [[a_1--a_2]////[b_1--b_2]]] & \text{by (OA.105)} \\
(A6) & \Rightarrow [[b_1--b_2]++++x]=[a_1--a_2] & \text{by (OA.80),(OA.104)}
\end{eqnarray*}

According to (OA.37)$\sim$(OA.40) and (OA.73), the function $ f(x)=[[b_1--b_2]++++x] $ is defined on the domain $ \bigl(-\infty, +\infty\bigl) $. Theorem \ref{Suffix} implies that we can iteratively decrease $u$ by step 1 from $u=1$ until $ [[b_1--b_2]++++u] < [a_1--a_2] $ holds. Theorem \ref{Suffix} also implies that we can iteratively increase $v$ by step 1 from $v=1$ until $ [a_1--a_2] < [[b_1--b_2]++++v] $ holds. 

Theorem \ref{Suffix} implies that $ f(x)=[[b_1--b_2]++++x] $ is continuous on the domain $ \bigl(-\infty, +\infty\bigl) $. In summary, both $ [[b_1--b_2]++++u] < [a_1--a_2] $ and $ [a_1--a_2] < [[b_1--b_2]++++v] $ hold.

Then Intermediate Value Theorem derives that the equation $ [[b_1--b_2]++++x]=[a_1--a_2] $ has only one root on the domain $ \bigl[u, v\bigl] $. Since Theorem \ref{Suffix} implies that the equation $ [[b_1--b_2]++++x]=[a_1--a_2] $ has no root on the domains $ \bigl(-\infty, u\bigl) $ and $ \bigl(v, +\infty\bigl) $, the equation $ [[b_1--b_2]++++x]=[a_1--a_2] $ has only one root on the domain $ \bigl[u, v\bigl] $. Since the equation $ [[b_1--b_2]++++x]=[a_1--a_2] $ belongs to Root Equations, Brent's method can find the only root of the equation and constructs the numerical computation for $ [[a_1--a_2]////[b_1--b_2]] $.
\end{enumerate}

\item If the numerical computations for $ [p_1+++\dot{e}p_2] $, $ [q_1---\dot{f}q_2] $, $ [r_1///\dot{g}r_2] $ have been constructed, then the numerical computations for $ [p_1++++\dot{e}p_2] $, $ [q_1----\dot{f}q_2] $, $ [r_1////\dot{g}r_2] $ can also be constructed.

According to (OA.19), the symbol `$e$' represents some successive `$+$'\textbf{---}``$+ \cdots +$". According to (OA.20), the symbol `$f$' represents some successive `$-$'\textbf{---}``$- \cdots -$". According to (OA.21), the symbol `$g$' represents some successive `$/$'\textbf{---}``$/ \cdots /$".

\begin{enumerate}\setlength{\parindent}{2em}
\item A numerical computation for $ [[a_1--a_2]++++\dot{e}[b_1--b_2]] $ with $ 1 \leq [a_1--a_2] $ and $ [1-1] \leq [b_1--b_2] $ can be constructed.
\begin{enumerate}\setlength{\parindent}{2em}
\item $ [b_1--b_2] \leq 1 $.
\begin{eqnarray*}
(A1) & (\bar{c} = [[[a_1--a_2]-1]++[1+[1+1]]]) \wedge  & \\
& (\bar{d} = [[[a_1--a_2]-1]++[1+1]]) \wedge \\
& (\bar{e} = [[b_1--b_2]+++[1+1]]) \wedge  & \\
& (\bar{f} = [[b_1--b_2]+++[1+[1+1]]]) \wedge & \\
& ([[a_1--a_2]++++\dot{h}[b_1--b_2]] = & \\
& [[1+[\bar{c}++[\bar{e}+++[1+\dot{k}]]]] - & \\
& [\bar{d}++[\bar{f}+++[1+\dot{k}]]]]) & \text{by (OA.83),(OA.29)} \\
\end{eqnarray*}

\item $ 1 < [b_1--b_2] $.
\begin{eqnarray*}
(B1) & [[a_1--a_2]++++\dot{e}[b_1--b_2]] = & \\
& [[a_1--a_2]+++\dot{e}[[a_1--a_2]++++\dot{e}[[b_1--b_2]-1]]] & \text{by (OA.82)} \\
(B2) & \text{It is supposed that $ [1-1] \leq [[b_1--b_2]-n] $ and } & \\
& \text{$ [[b_1--b_2]-n] \leq 1 $. Let us distinguish these } & \\
& \text{$[a_1--a_2]$ with the subscripts } \{(1), (2), (3), \cdots\}. & \\
(B3) & \Rightarrow [[a_1--a_2]++++\dot{e}[b_1--b_2]] = & \\
& [[a_1--a_2]_{(1)}+++\dot{e}[[a_1--a_2]_{(2)}+++\dot{e} \cdots \\
& [[a_1--a_2]_{n}+++\dot{e}
[[a_1--a_2]++++\dot{e} & \\
& [[b_1--b_2]-n]] \cdots ]] & \text{by (OA.82),(B2)} \\
(B4) & (\bar{c} = [[[a_1--a_2]-1]++[1+[1+1]]]) \wedge  & \\
& (\bar{d} = [[[a_1--a_2]-1]++[1+1]]) \wedge \\
& (\bar{e} = [[[b_1--b_2]-n]+++[1+1]]) \wedge  & \\
& (\bar{f} = [[[b_1--b_2]-n]+++[1+[1+1]]]) \wedge & \\
& ([[a_1--a_2]++++\dot{e}[[b_1--b_2]-n]] = & \\
& [[1+[\bar{c}++\bar{e}]]-[\bar{d}++\bar{f}]]) & \text{by (B3),(A1)} \\
(B5) & \Rightarrow (\bar{c} = [[[a_1--a_2]-1]++[1+[1+1]]]) \wedge  & \\
& (\bar{d} = [[[a_1--a_2]-1]++[1+1]]) \wedge \\
& (\bar{e} = [[[b_1--b_2]-n]+++[1+1]]) \wedge  & \\
& (\bar{f} = [[[b_1--b_2]-n]+++[1+[1+1]]]) \wedge & \\
& ([[a_1--a_2]++++\dot{e}[b_1--b_2]] = & \\
& [[a_1--a_2]_{(1)}+++\dot{e}[[a_1--a_2]_{(2)}+++\dot{e} \cdots \\
& [[a_1--a_2]_{n}+++\dot{e}
[[1+[\bar{c}++\bar{e}]]-[\bar{d}++\bar{f}]] \cdots ]]) & \text{by (B3),(B4)} \\
\end{eqnarray*}

Then (A1) and (B1)$\sim$(B5) have reduced one $ ++++\dot{e} $ operation to many $ +++\dot{e} $ operations. Since the numerical computation for $ [[a_1--a_2]+++\dot{e}[b_1--b_2]] $ has been supposed to be constructed, (A1) and (B1)$\sim$(B5) can achieve a numerical computation for $ [[a_1--a_2]++++\dot{e}[b_1--b_2]] $.
\end{enumerate}

\item A numerical computation for $ [[a_1--a_2]----\dot{f}[b_1--b_2]] $ with $ 1 \leq [a_1--a_2] $ and $ [1-1] < [b_1--b_2] $ can be constructed.

The numerical computation for $ [[a_1--a_2]----\dot{f}[b_1--b_2]] $ is equated with the numerical root finding of the equation $ x=[[a_1--a_2]----\dot{f}[b_1--b_2]] $.
\begin{eqnarray*}
(A1) & x=[[a_1--a_2]----\dot{f}[b_1--b_2]] & \\
(A2) & \Rightarrow [x++++\dot{e}[b_1--b_2]]= & \\
& [[[a_1--a_2]----i[b_1--b_2]] & \\
& ++++h[b_1--b_2]] & \text{by (OA.105)} \\
(A3) & \Rightarrow [x++++\dot{e}[b_1--b_2]]=[a_1--a_2] & \text{by (OA.78),(OA.24),} \\ 
&& \text{(OA.25)}
\end{eqnarray*}

According to (OA.37)$\sim$(OA.40), the function $ f(x)=[x++++\dot{e}[b_1--b_2]] $ is defined on the domain $ \bigl[1, +\infty\bigl) $. Theorem \ref{Prefix} implies that we can iteratively increase $v$ by step 1 from $v=1$ until $ [a_1--a_2] < [v++++\dot{e}[b_1--b_2]] $ holds. 

Theorem \ref{Prefix} implies that $ f(x)=[x++++\dot{e}[b_1--b_2]] $ is continuous on the domain $ \bigl[1, v\bigl] $. (OA.75) derives that $ [1++++\dot{e}[b_1--b_2]] = 1 $. So (OA.95) derives that $ [1++++\dot{e}[b_1--b_2]] \leq [a_1--a_2] $. In summary, both $ [1++++\dot{e}[b_1--b_2]] \leq [a_1--a_2] $ and $ [a_1--a_2] < [v++++\dot{e}[b_1--b_2]] $ hold.

Then Intermediate Value Theorem derives that the equation $ [x++++\dot{e}[b_1--b_2]]=[a_1--a_2] $ has only one root on the domain $ \bigl[1, v\bigl] $. Since Theorem \ref{Prefix} implies that the equation $ [x++++\dot{e}[b_1--b_2]]=[a_1--a_2] $ has no root on the domain $ \bigl(v, +\infty\bigl) $, the equation $ [x++++\dot{e}[b_1--b_2]]=[a_1--a_2] $ has only one root on the domain $ \bigl[[1-1], +\infty\bigl) $. Since the equation $ [x++++\dot{e}[b_1--b_2]]=[a_1--a_2] $ belongs to Root Equations, Brent's method can find the only root of the equation and constructs the numerical computation for $ [[a_1--a_2]----\dot{f}[b_1--b_2]] $.

\item A numerical computation for $ [[a_1--a_2]////\dot{g}[b_1--b_2]] $ with $ 1 \leq [a_1--a_2] $ and with $ 1 < [b_1--b_2] $ can be constructed.

The numerical computation for $ [[a_1--a_2]////\dot{g}[b_1--b_2]] $ is equated with the numerical root finding of the equation $ x=[[a_1--a_2]////\dot{g}[b_1--b_2]] $.
\begin{eqnarray*}
(A1) & x=[[a_1--a_2]////\dot{g}[b_1--b_2]] & \\
(A2) & [1-1] < [[a_1--a_2]////\dot{g}[b_1--b_2]] & \text{by (OA.33)} \\
(A3) & \Rightarrow 1 < [[b_1--b_2]++++h & \\
& [[a_1--a_2]////j[b_1--b_2]]] & \text{by (OA.37),(OA.72)} \\
(A4) & \Rightarrow ([[b_1--b_2]++++h & \\
& [[a_1--a_2]////j[b_1--b_2]]]) & \text{by (OA.90)} \\
(A5) & \Rightarrow [[b_1--b_2]++++\dot{e}x]= & \\
& [[b_1--b_2]++++h & \\
& [[a_1--a_2]////j[b_1--b_2]]] & \text{by (OA.105)} \\
(A6) & \Rightarrow [[b_1--b_2]++++\dot{e}x]=[a_1--a_2] & \text{by (OA.80),(OA.104)}
\end{eqnarray*}

According to (OA.37)$\sim$(OA.40) and (OA.73), the function $ f(x)=[[b_1--b_2]++++\dot{e}x] $ is defined on the domain $ \bigl(-\infty, +\infty\bigl) $. Theorem \ref{Suffix} implies that we can iteratively decrease $u$ by step 1 from $u=1$ until $ [[b_1--b_2]++++\dot{e}u] < [a_1--a_2] $ holds. Theorem \ref{Suffix} also implies that we can iteratively increase $v$ by step 1 from $v=1$ until $ [a_1--a_2] < [[b_1--b_2]++++\dot{e}v] $ holds. 

Theorem \ref{Suffix} implies that $ f(x)=[[b_1--b_2]++++\dot{e}x] $ is continuous on the domain $ \bigl(-\infty, +\infty\bigl) $. In summary, both $ [[b_1--b_2]++++\dot{e}u] < [a_1--a_2] $ and $ [a_1--a_2] < [[b_1--b_2]++++\dot{e}v] $ hold.

Then Intermediate Value Theorem derives that the equation $ [[b_1--b_2]++++\dot{e}x]=[a_1--a_2] $ has only one root on the domain $ \bigl[u, v\bigl] $. Since Theorem \ref{Suffix} implies that the equation $ [[b_1--b_2]++++\dot{e}x]=[a_1--a_2] $ has no root on the domains $ \bigl(-\infty, u\bigl) $ and $ \bigl(v, +\infty\bigl) $, the equation $ [[b_1--b_2]++++\dot{e}x]=[a_1--a_2] $ has only one root on the domain $ \bigl[u, v\bigl] $. Since the equation $ [[b_1--b_2]++++\dot{e}x]=[a_1--a_2] $ belongs to Root Equations, Brent's method can find the only root of the equation and constructs the numerical computation for $ [[a_1--a_2]////\dot{g}[b_1--b_2]] $.
\end{enumerate}

\item By induction, the numerical computations for $ [p_1+++++p_2] $, $ [p_1++++++p_2] $, $ [p_1+++++++p_2] $, $ \cdots $, $ [q_1-----q_2] $, $ [q_1------q_2] $, $ [q_1-------q_2] $, $ \cdots $, $ [r_1/////r_2] $, $ [r_1//////r_2] $, $ [r_1///////r_2] $, $ \cdots $ can all be constructed.
\end{enumerate}

\section{Some Concepts In The Operator Axioms}
We define some replacements for the notations of the Operator axioms, as is shown in \autoref{table:definitions}.

\renewcommand\arraystretch{1.5}
\begin{table}[h]
\begin{center}
\centering \caption{Replacements Of Notations.}
\label{table:definitions}
\begin{tabular}{|c|c|c|}
\hline
\textbf{Replacement} & \textbf{Notation In Operator axioms} \\
\hline
$+'$ & $\dot{e}$ \\
\hline
$-'$ & $\dot{f}$ \\
\hline
$/'$ & $\dot{j}$ \\
\hline
$+'_1$ & $+$ \\
\hline
$+'_2$ & $++$ \\
\hline
$+'_3$ & $+++$ \\
\hline
$+'_n$ & $\underbrace{++\cdots+}_{n}$ \\
\hline
$-'_1$ & $-$ \\
\hline
$-'_2$ & $--$ \\
\hline
$-'_3$ & $---$ \\
\hline
$-'_n$ & $\underbrace{--\cdots-}_{n}$ \\
\hline
$/'_1$ & $/$ \\
\hline
$/'_2$ & $//$ \\
\hline
$/'_3$ & $///$ \\
\hline
$/'_n$ & $\underbrace{//\cdots/}_{n}$ \\
\hline
\end{tabular}
\end{center}
\end{table}

The notation $+'$ can be replaced by any element of the set \{$ +, ++, +++, ++++, \cdots $\}. The notation $-'$ can be replaced by any element of the set \{$ -, --, ---, ----, \cdots $\}. The notation $/'$ can be replaced by any element of the set \{$ /, //, ///, ////, \cdots $\}.

We define the pronunciations for some expressions in the Operator axioms, as is shown in \autoref{table:pronunciations}.

\renewcommand\arraystretch{1.5}
\begin{table}[h]
\begin{center}
\centering \caption{Pronunciations For Some Expressions.}
\label{table:pronunciations}
\begin{tabular}{|c|c|c|}
\hline
\textbf{Expression} & \textbf{Pronunciation} \\
\hline
$+'$ & addote \\
\hline
$-'$ & subote \\
\hline
$/'$ & logote \\
\hline
$a +'_n b$ & a addote n to b \\
\hline
$a -'_n b$ & a subote n to b \\
\hline
$a /'_n b$ & a logote n to b \\
\hline
\end{tabular}
\end{center}
\end{table}

We divide the real operators into an ordered level with the natural numbers. \autoref{table:levels} lists the levels of the real operators in detail.

\begin{table}[h]
\begin{center}
\centering \caption{Level Of Operators}
\label{table:levels}
\begin{tabular}{|c|c|c|}
\hline
\textbf{Level} & \textbf{Operators} \\
\hline
\text{1} & $+'_1, -'_1, /'_1$ \\
\hline
\text{2} & $+'_2, -'_2, /'_2$ \\
\hline
\text{3} & $+'_3, -'_3, /'_3$ \\
\hline
\text{$ \cdots $} & $ \cdots $ \\
\hline
\text{n} & $+'_n, -'_n, /'_n$ \\
\hline
\end{tabular}
\end{center}
\end{table}

The order of real operators is listed as follows:

level-1 $<$ level-2 $<$ level-3 $<\cdots$ $<$ level-n 

We define the operations of real operators as follows.
\begin{definition}\label{Complete}
\emph{Complete Operations} are all the binary operations of real operators.
\end{definition}

According to the Definition \ref{Complete}, all operations such as $a +'_n b$, $a -'_n b$, $a /'_n b$ compose the complete operations.

\clearpage

%

\end{document}